\documentclass[12pt]{amsart}
\usepackage{amscd}
\usepackage{amsmath}
\usepackage{amsxtra}
\usepackage{amsfonts}
\usepackage{amssymb}

\usepackage{amsthm}
\usepackage{amssymb}
\usepackage{enumerate}
\usepackage{graphicx}
\usepackage{amsmath}
\usepackage{subfigure}
\usepackage[linkcolor=red,anchorcolor=blue,citecolor=green]{hyperref}
\usepackage{color}

\hypersetup{
     colorlinks,%
    citecolor=blue,%
    filecolor=black,%
    linkcolor=magenta,%
    urlcolor=black
}

 \newtheorem{theorem}{Theorem}[section]
 \newtheorem{Def}[theorem]{Definition}
 \newtheorem{Prop}[theorem]{Proposition}
 \newtheorem{Lem}[theorem]{Lemma}

 \newtheorem{Example}[theorem]{Example}

\newcommand{\R}{{\mathbb R}}
\newcommand{\Z}{{\mathbb Z}}

\newcommand{\N}{{\mathbb N}}

\title{Undersampled windowed exponentials and their applications}

\author{Chun-Kit Lai}

\address{Department of Mathematics, San Francisco State University,
1600 Holloway Avenue, San Francisco, CA 94132.}

 \email{cklai@sfsu.edu}

\author{Sui Tang}
\address{Department of Mathematics, Johns Hopkins University,
3400 North Charles Street, Baltimore, MD, 21218.}

\email{stang@math.jhu.edu}


\begin{document}
\maketitle

\begin{abstract}
We characterize the completeness and frame/basis property of a union of under-sampled windowed exponentials of the form
$$
{\mathcal F}(g): =\{e^{2\pi i n x}: n\ge 0\}\cup \{g(x)e^{2\pi i nx}: n<0\}
$$
for $L^2[-1/2,1/2]$ by the spectra of the Toeplitz operators with the symbol $g$. Using this characterization, we classify all real-valued functions $g$ such that ${\mathcal F}(g)$ is complete or forms a frame/basis. Conversely, we use the classical non-harmonic Fourier series theory to determine all $\xi$ such that the Toeplitz operators with symbol $e^{2\pi i \xi x}$ is injective or invertible. These results demonstrate an elegant interaction between frame theory of windowed exponentials and Toeplitz operators. Finally, we use our results to answer some open questions in dynamical sampling,  and derivative samplings on Paley-Wiener spaces of bandlimited functions. 

 \end{abstract}


\medskip

\section{Introduction}

\noindent{\bf Background.} Let $\Omega$ be a measurable set of finite Lebesgue measure in ${\mathbb R}^d$. Suppose that  $g\in L^2(\Omega)\setminus\{0\}$ and $\Lambda$ is a countable subset of ${\mathbb R}^d$. We define the collection of {\it windowed exponentials} by
$$
{\mathcal E}(g,\Lambda) : = \{g(x)e^{2\pi i \langle\lambda ,x\rangle}: \lambda\in\Lambda\}.
$$
A natural question is to determine when a union of finitely many windowed exponentials $\bigcup_{j=1}^{N}{\mathcal E}(g_j,\Lambda_j)$ is complete or forms a frame/Riesz basis for $L^2(\Omega)$.  Let us recall the definitions as below:

\begin{Def}
{\rm The collection of windowed exponentials $\bigcup_{j=1}^{N}{\mathcal E}(g_j,\Lambda_j)$ forms a {\it frame} for $L^2(\Omega)$ if there exists $0<A\le B<\infty$ such that}
$$
A\|f\|^2\le \sum_{j=1}^N\sum_{\lambda\in\Lambda_j}\left|\int f(x)\overline{g_j(x)}e^{-2\pi i \langle\lambda,x\rangle}dx\right|^2\le B\|f\|^2.
$$
{\rm It is called a \it{Riesz basis} if the collection forms a frame with the property that every $f\in L^2(\Omega)$ is expanded uniquely as  $\sum \alpha_{j,\lambda} g_j(x)e^{2\pi i \langle\lambda,x\rangle}$. It is called minimal if no single windowed exponential  of $\bigcup_{j=1}^{N}{\mathcal E}(g_j,\Lambda_j)$ lies in the $L^2$ span of the other  windowed exponentials of $\bigcup_{j=1}^{N}{\mathcal E}(g_j,\Lambda_j)$. } 

\medskip

{\rm The collection of windowed exponentials is called {\it complete} if  the equations}
$$
\int_{\Omega} f(x)\overline{g_j(x)}e^{-2\pi i \langle\lambda,x\rangle}dx = 0,
  \ \forall   \lambda \in\Lambda_j, j=1,..,N,
$$
{\rm then imply that $f= 0$ a.e. on $\Omega$.}

\medskip

{\rm We say that $\{g_1,...,g_N\}$ is {\it admissible} if there exists some $\Lambda_j$ such that $\bigcup_{j=1}^N{\mathcal E}(g_j,\Lambda_j)$ forms a frame of windowed exponentials for $L^2(\Omega)$.}
\end{Def}
\medskip

The question in determining the completeness and frame/basis properties of windowed exponentials is closely related to wavelet theory and Gabor analysis since problems about the frame of translate in the multi-resolution analysis and regular Gabor system can be reduced to those of windowed exponentials via respectively the Fourier transform and the Zak transform.  There is rich literature in this direction.  One can refer to \cite{christensen2003introduction, heil2007history, heil2010basis} for an introduction and \cite{balan2006density, heil2008density, heil2012duals} for some recent work about windowed exponentials. 
\medskip

We also notice that if there is only one window, namely $g = \chi_{\Omega}$, then if ${\mathcal E}(\chi_{\Omega},\Lambda)$  is called a {\it Fourier frame} if it forms a frame.  Introduced by Duffin and Schaeffer \cite{duffin1952class} in the 1950s,  Fourier frames have known to be a fundamental building block of sampling theory in the applied harmonic analysis, as then 
functions with frequency compactly supported in $\Omega$ can be stably recovered from their samples on 
$\Lambda$.  The starting point of Fourier frames can be traced back to Paley-Wiener, Levinson, Plancherel-Polya, and Boas for studying the possibility of non-harmonic series expansion \cite{levinson1940gap, redheffer1983completeness, seip1995connection}. We refer to \cite{adcock2015weighted, adcock2017computing, aldroubi2001nonuniform,  benedetto1992irregular,casazza2006density, grochenig1999irregular, jaffard1991density,  ortega2002fourier} and the reference therein for some recent advances and applications in sampling theory.

\medskip

In general, it is well-known that every bounded set $\Omega \subset \R^d$ admits a Fourier frame. Indeed, if we cover $\Omega$ by a cube, then the exponential orthonormal basis with frequency set denoted by $\Lambda$ on the cube induces a tight Fourier frame for $L^2(\Omega)$. More generally, if $0<m\le g\le M<\infty$ on $\Omega$, then ${\mathcal E}(g,\Lambda)$ forms a frame for $L^2(\Omega)$.  In \cite{gabardo2014frames} (See also \cite{lai2011fourier}), the class of all  admissible windows $\{g_1,...,g_N\}$ were completely classified. The necessary and sufficient condition for $\{g_1,...,g_N\}$ to be admissible for $L^2(\Omega)$ is that there exists $c>0$ with the property that
$$
\max_{\{j: \|g_j\|_{\infty}<\infty\}} |g_j|\ge c>0. \ \mbox{a.e. on} \ \Omega
$$
In particular, when there is only one window $g$, $g$ is admissible if and only if $0<m\le g\le M$ a.e. on $\Omega$.

\medskip

If $\{g_1,...,g_N\}$ is admissible, it is known that (see \cite{gabardo2014frames}) one can produce a frame of multi-windowed exponentials by taking $\Lambda_j$ in an oversampling manner in the sense that the (Beurling) density of $\Lambda_j$ are strictly greater than the measure of $\Omega$. 

\medskip

Despite the success of determining window exponentials using oversampled $\Lambda_j$,  there are fewer papers studying the case in which each $\Lambda_j$ is indeed undersampled or even of Beurling density zero, yet a union of windowed exponentials $\bigcup_{j=1}^{N}{\mathcal E}(g_j,\Lambda_j)$ forms a frame and works perfectly.  This paper considers the possibility of forming a frame in the undersampling circumstances.  We focus on the following case  let ${\mathbb N}_0={\mathbb N^+}\cup \{0\}$ and ${\mathbb N}^{-}$ be the set of all negative integers and define
\begin{equation}\label{window1}
{\mathcal F}(g) = {\mathcal E}(\chi_{[-1/2,1/2]}, {\mathbb N}_0)\cup {\mathcal E}(g, {\mathbb N}^{-}).
\end{equation}
We ask for what $g$  the set of functions ${\mathcal F}(g)$  is complete/ or forms a Riesz basis/frame for $L^2[-1/2,1/2]$.
 The consideration of the completeness properties ${\mathcal F}(g)$ has immediate applications to the dynamical sampling theory and the nonuniform derivative sampling theory, which we will discuss in a later section.

\medskip

\noindent{\bf Main Results.} It turns out the solution to the question above can be characterized completely by the invertibility of the Toeplitz operators. While the connection between Toeplitz operator and windowed Fourier frame  should have  been known in the frame theory community (see Theorem 2.1 of {\color{red} Casazza,Christensen Kalton} paper with $\Lambda = {\mathbb N}$), this paper used it rigorously, probably for the first time, to study the frame propoerty of a class of windowed exponentials, ${\mathcal F}(g)$.   First,  let us review some known facts about Toeplitz operators. 
\begin{Def}
{\rm We say that a bounded linear operator $T:\ell^2({\mathbb N}_0)\rightarrow\ell^2({\mathbb N}_0)$ is a {\it Toeplitz operator} if, with respect to the standard basis of $\ell^2({\mathbb N}_0)$, $T$ admits a matrix representation}
\begin{equation}\label{eqT}
T =
\left[
  \begin{array}{cccc}
    a_0&a_{-1}& a_{-2}&\cdots\\
    a_{1}&a_0 & a_{-1} & \ddots \\
    a_{2}&a_{1}& \ddots & \ddots \\
    \vdots&\ddots & \ddots&  \\
  \end{array}
\right]
\end{equation}
\end{Def}

 {\rm Let $1 \leq p \leq \infty$ and $\tilde{H}^p$ be a closed subspace of  $L^p[-1/2,1/2]$  defined by} $$\tilde{H}^p=\{\tilde{ f}\in L^p[-1/2,1/2]:  \langle f,e^{2\pi i nx} \rangle=0 \text{ for } n \in \N^-\}.$$

{\rm Let $\phi\in L^{\infty}([-1/2,1/2])$ and let $P_+:L^2[-1/2,1/2] \rightarrow \tilde{H}^2 $ be the Riesz projection operator. Then we can define  a bounded linear operator $T_{\phi}:\tilde{H}^2 \rightarrow \tilde{H}^2 $ by}
$$
T_\phi f : =  P_+(\phi f).
$$
{\rm  $T_{\phi}$ is called the Topelitz operator with symbol $\phi$ and $T_{\phi}$ is bounded. The matrix representation of $T_{\phi}$ in terms of the standard basis $\{e^{2\pi i nx}: n\in \N_0\}$ of $\tilde{H}^2$  is given by
 $$
T_{\phi} =
\left[
  \begin{array}{cccc}
    \widehat{\phi}(0)&\widehat{\phi}(-1)& \widehat{\phi}(-2)&\cdots\\
    \widehat{\phi}(1)&\widehat{\phi}(0) & \widehat{\phi}(-1) & \ddots \\
    \widehat{\phi}(2)&\widehat{\phi}(1)& \ddots & \ddots \\
    \vdots&\ddots & \ddots&  \\
  \end{array}
\right].
 $$
  Conversely, if $T$ is a bounded Toeplitz operator with representation \eqref{eqT},  then there exists $\phi\in L^\infty [-1/2,1/2]$ such that $T$ is the matrix representation of  $T_{\phi}$. Furthermore, we have $\|T_{\phi}\| = \|\phi\|_{\infty}$}

\medskip

The following theorem is our main result in characterizing the basis/frame properties of ${\mathcal F}(g)$  by the invertibility of the Toeplitz operators $T_g$.

\begin{theorem}\label{thmain1}
\begin{enumerate}
\item  ${\mathcal F}(g)$ is complete if and only if  $T_g$ is injective.
\item ${\mathcal F}(g)$  forms a frame if and only if $T_g$ is both norm-bounded above and below.
\item ${\mathcal F}(g)$ forms a Riesz basis if and only if $T_g$ is bounded and invertible.
\end{enumerate}
\end{theorem}

\medskip

We will prove this theorem in Theorem \ref{th2}. This theorem gives us in principle a complete characterization of the frame and basis property of ${\mathcal F}(g)$ regarding the invertibility of the Toeplitz operator and the triviality of the Toeplitz kernel. This leads us to look for classical results on Toeplitz operators. In particular, we need theorems characterizing the frame property by the function values of $g$.    A complete classification for the invertible Toeplitz operators $T_g$ was given {\it analytically} through the Widom-Devinatz theorem in the 1960s.  We refer the reader to \cite{bottcher2013analysis, martinez2007introduction} for the detailed study.   Unfortunately, as far as we know, there is no easily checkable  criterion to check if the Toeplitz operator $T_g$ is invertible for a general complex-valued window $g$.

\medskip

In our paper, we will provide a complete characterization of the completeness and frame/bases property of ${\mathcal F}(g)$ when $g$ is real-valued (see Theorem \ref{threalg}).  After that, we will consider the complex windows $g_{\xi}(x) = e^{2\pi i \xi x}$ for $\xi \in \mathbb{R}$.  While from Toeplitz operator side, there is no easily checkable way to verify the invertibility of  $T_{g_{\xi}}$, we use some theories of non-harmonic Fourier series and completely characterize all $\xi$ for which ${\mathcal F}(g_{\xi})$ is complete/ a frame or a Riesz basis (Theorem \ref{th4}). Formulated in the Toeplitz operators, the following theorem may be of independent interest.

\medskip

\begin{theorem}\label{th0.2} Let $g_{\xi}(x) = e^{2\pi i \xi x}$ and $\xi \in \mathbb{R}$. Then
\begin{enumerate}
\item  $T_{g_{\xi}}$ is not injective if $\xi<-1/2$.
\item  $T_{g_{\xi}}$ is invertible if $-1/2<\xi<1/2$.
\item  $T_{g_{\xi}}$ is injective but  not normed-bounded below if $\xi =-1/2+n$ for $n=0,1,2,...$.
\item $T_{g_{\xi}}$ is injective and normed-bounded below but  not surjective if  $\xi>1/2$ and $\xi \ne -1/2+n$ for $n=0,1,2,...$
\end{enumerate}
\end{theorem}
This theorem will follow directly from Theorem \ref{th4} and Theorem \ref{thmain1}. These results demonstrate an elegant interaction between windowed exponentials and Toeplitz operators. 
 
\medskip

We organize our paper as follows: In Section 2, we will give some notations and preliminaries for the Toeplitz operators required for the rest of the papers. We will then prove Theorem \ref{thmain1} in Section 3 followed by classification of real-valued windows and complex exponential windows for ${\mathcal F}(g)$ to be a frame/basis.  In Section 3, we will consider the applications of our classification in dynamical sampling and nonuniform derivative sampling.
We conclude our paper in Section 4.

\section{Preliminaries and Notation}
We now introduce some notation and preliminaries that are useful in the following sections.  Let $A: H\rightarrow H$ be a bounded linear operator on a  Hilbert space $H$. We use  $\sigma(A)$ to denote the spectrum of $A$ which is the set of complex numbers $\lambda$ such that $A-\lambda I$ is not invertible. $\sigma(A)$ contains several types of spectra; $\sigma_p(A)$ denotes the {\it point spectrum} of $A$ consisting of all $\lambda$ such that $A-\lambda I$ is not injective; $ \sigma_{ap}(A)$ denotes the {\it approximate point spectrum} consisting of all $\lambda$ such that  there exists $f_n\in H$ with $\|f_n\|=1$ such that $\|(A-\lambda I)f_n\|\rightarrow 0$; $\sigma_c(A)$ denotes the {\it compression spectrum} consisting of points  $\lambda$ such that $A-\lambda I$ does not have a dense range.  For every bounded operator $A$, we have 
\begin{itemize}
\item $\sigma(A)=\sigma_{ap}(A) \cup \sigma_c(A)$, $\sigma_{ap}(A) \cap  \sigma_c(A)$ may not be empty. 
\item $\sigma_{p}(A) \subset \sigma_{ap}(A)$.
\item the boundary of $\sigma(A)$ denoted by $\partial \sigma(A) \subset \sigma(A)$.  
\end{itemize}

For $\phi\in L^\infty$, the {\it essential range} of $\phi$ is defined as
$$
\mbox{essran} (\phi) : = \{y: m(\{x:|\phi(x)-y|<\epsilon \})>0 \ \mbox{for all} \ \epsilon>0\}.
$$
The {\it essential infimum} and {\it essential supremum} of a real-valued function $\phi$ are defined as
$$
\mbox{essinf} (\phi) : = \inf \mbox{essran} (\phi), \ \mbox{esssup} (\phi) : = \sup \mbox{essran} (\phi).
$$

The following theorem is useful in proving. our theorems, and can be found in \cite{martinez2007introduction}.

\begin{theorem}\label{thMR}
Let $T_{\phi}$ be the Toeplitz operator associated with the $L^{\infty}$ function $\phi$. Then

\medskip

\begin{enumerate}
\item ${\mbox essran}(\phi)\subset \sigma_{ap}(T_{\phi})$.
\item (Coburn alternative) $T_{\phi}$ or $T_{\overline{\phi}}$ is injective. In particular, if $\phi$ is real-valued, then $T_{\phi} $ is injective.
\item Suppose that $\phi$ is real-valued. Then $\sigma(T_{\phi}) = [\mbox{essinf}(\phi),\mbox{esssup}(\phi)]$. Moreover, for all $\lambda$ in the open interval $ (\mbox{essinf}(\phi),\mbox{esssup}(\phi))$,  $T_{\phi}-\lambda I$ is not surjective.
\end{enumerate}
\end{theorem}

The last part of the theorem follows from the proof in \cite[Theorem 3.3.15]{martinez2007introduction}.  Finally, we mention the general analytic invertibility theorem for $T_{\phi}$ due to Widom-Devinatz (see e.g., \cite[Theorem 2.23]{bottcher2013analysis}).

 \medskip

 \begin{theorem}\label{WD} [Widom-Devinatz]
 $T_{\phi}$ is invertible if and only if $\phi^{-1}\in L^{\infty}[-1/2,1/2]$ and
 \begin{equation}\label{H-S1}
 \frac{\phi}{|\phi|} = e^{i(u^h+v+c)} \ \mbox{a.e. on} \ [-1/2,1/2]
 \end{equation}
 where $c\in{\mathbb R}$, $u,v$ are bounded real-valued functions, $\|v\|_{\infty}<\pi/2$ and ${u}^h$ is the harmonic conjugate of $u$, i.e.,  it satisfies the property that  $u+i{u}^h$ can be extended analytic in the unit disk.
 \end{theorem}

 Condition  in (\ref{H-S1}) is commonly referred as {\it Helson-Szeg\"{o} condition}.

\section{Frame properties of ${\mathcal F}(g)$}
In this section, we will use the spectral properties of Toeplitz operators to characterize the frame properties of ${\mathcal F}(g)$. Recall that
$$
{\mathcal F}(g) = {\mathcal E}(\chi_{[-1/2,1/2]}, {\mathbb N}_0)\cup {\mathcal E}(g, {\mathbb N}^{-}).
$$
 ${\mathcal F}(g)$ forms a frame if and only if there exist $0<A\le B<\infty$ such that
$$
A\|f\|^2\le \sum_{n=-\infty}^{-1}|\langle f,ge_n\rangle|^2+ \sum_{n=0}^{\infty}|\langle f,e_n\rangle|^2\le B\|f\|^2,
$$
for all $f\in L^2[-1/2,1/2]$. It forms a {\it Bessel sequence} if  the second inequality above holds.
\medskip

We now  consider the {\it analysis operator} of  ${\mathcal F}(g)$, $
\Phi_g: L^2\left[-\frac12,\frac12\right] \longrightarrow \ell^2({\mathbb Z})$ defined by
$$
\Phi_g f =  \left(\cdots,\langle f,ge_{-2}\rangle,\langle f,ge_{-1}\rangle,\langle f,e_0\rangle, \langle f,e_1\rangle,\langle f,ge_2\rangle\cdots\right).
$$
The following theorem is well-known in the frame theory literature (See e.g \cite[Theorem 8.29 and 8.32]{ heil2010basis}).

\begin{theorem}\label{th1}
\begin{enumerate}
\item ${\mathcal F}(g)$ is complete if and only if $\Phi_g$ is injective.
\item ${\mathcal F}(g)$ is a frame if and only if $\Phi_g$ is both norm-bounded above and below.
\item ${\mathcal F}(g)$ is a Riesz basis if and only if $\Phi_g$ is surjective and both norm-bounded above and below.
\end{enumerate}
\end{theorem}

Note that $L^2[-\frac12,\frac12] $ is unitarily equivalent to $\ell^2({\mathbb Z})$ via the Fourier coefficients.
$$
f\in L^2\left[-\frac12,\frac12\right]  \longleftrightarrow \{\widehat{f}(n)\}_{n\in{\mathbb Z}}.
$$
We denote $e_n(x)= e^{2\pi i n x}$. Suppose that $g = \sum_{n=-\infty}^{+\infty} b_n e_n$. Then
$$
\overline{g} = \sum_{n=-\infty}^{+\infty} \overline{b_{-n}} e_n.
$$
Hence,
$$
\langle f,ge_{n}\rangle = \langle f\overline{g},e_{n}\rangle = \sum_{k=-\infty}^{+\infty}  \overline{b_{-(n-k)}}\widehat{f}(k).
$$
We also define the function
$$
\widetilde{g}(x) = \overline{g(-x)} = \sum_{n=-\infty}^{\infty}\overline{b_{n}}e_{n}(x).
$$
Suppose that we identify ${\mathbb Z} = {\mathbb N}^{-}\oplus{\mathbb N}_0$ and order ${\mathbb N}^{-} = \{-1,-2,...\}$ and ${\mathbb N}_0 = \{0,1,2,...\}$. We can think of $\Phi_g$ as a mapping from $\ell^2(\mathbb{ N}^{-}) \oplus \ell^2(\mathbb{ N}_{0} ) $ to $\ell^2(\mathbb{ N}^{-}) \oplus \ell^2(\mathbb{ N}_{0} ) $, the map becomes
$$
\Phi_g(a_{-1}, \cdots, a_0,a_1,\cdots) = (\sum_{k=-\infty}^{+\infty}  \overline{b_{k+1}}a_k,\sum_{k=-\infty}^{+\infty}  \overline{b_{k+2}}a_k,\cdots,a_0,a_1,\cdots)
$$
Hence, $\Phi_g$ admits a matrix representation of the form
\begin{equation}\label{rep_Phi}
\left[
  \begin{array}{ccc}
    T_{\widetilde{g}} & | & H_{\overline{g}} \\
    -- & | & -- \\
    O & |& I \\
  \end{array}
\right]
\end{equation}
where $T_{\widetilde{g}}: \ell^2({\mathbb N}^-)\longrightarrow \ell^2({\mathbb N}^-)$ and $H_{ \overline{g}}: \ell^2({\mathbb N}_0)\longrightarrow\ell^2({\mathbb N}^{-})$ with the matrix representation
$$
T_{\widetilde{g}} =
\left[
  \begin{array}{cccc}
    \overline{b_0}&\overline{b_{-1}}&\overline{b_{-2}} &\cdots\\
    \overline{b_1}&\overline{b_0} & \overline{b_{-1}} &\ddots \\
    \overline{b_2}&\overline{b_{1}}& \ddots & \ddots \\
    \vdots&\ddots & \ddots&  \\
  \end{array}
\right], \  H_{\overline{g}}  = \left[
  \begin{array}{cccc}
    \overline{b_{1}}&\overline{b_2} & \overline{b_3}& \cdots  \\
    \overline{b_2}&\overline{b_{3}}& \overline{b_4} & \ddots \\
    \overline{b_3}&\overline{b_{4}} & \ddots& \ddots  \\
    \vdots&\ddots& \ddots &
  \end{array}
\right].
$$

With a slight abuse of notation, we identify $\ell^2({\mathbb N}^{-})$  and  $\ell^2({\mathbb N}_0)$ by an obvious isometric isomorphism, so that $T_{\widetilde{g}}$ is a Toeplitz operator. We have a simple lemma concerning the spectra of $T_{\widetilde{g}}$ and $T_g$. 


\medskip

\begin{Lem}\label{lem1+}
$\sigma_p(T_{\widetilde{g}}) = \overline{\sigma_p(T_{{g}})}$, $\sigma_{ap}(T_{\widetilde{g}}) = \overline{\sigma_{ap}(T_{{g}})}$ and $\sigma_c(T_{\widetilde{g}}) = \overline{\sigma_c(T_{{g}})}$.

\smallskip

\end{Lem}

\begin{proof}
This lemma follows from the a direct observation that for any $f = (f_1,f_2,...)$, $T_{\widetilde{g}} f = \overline{T_g \overline{f}}$.
\end{proof}

\medskip

We need the following general lemma concerning the boundedness of operators on $\ell^2({\mathbb Z})$ and $\ell^2({\mathbb N})$.

\smallskip

\begin{Lem}\label{main_lemma}
Let $\Phi: \ell^2({\mathbb Z})\to \ell^2({\mathbb Z})$ be a linear operator such that $\Phi$ has a matrix representation of the form
$$
\left[
  \begin{array}{ccc}
    A & | & B \\
    -- & | & -- \\
    O & |& I \\
  \end{array}
\right],
$$
where $A: \ell^2({\mathbb N}^{-})\to\ell^2({\mathbb N})$ and $B: \ell^2({\mathbb N}^{-})\to\ell^2({\mathbb N}_0)$ are bounded linear operator. Then
\begin{enumerate}
\item $\Phi$ is a bounded linear operator on $\ell^2({\mathbb Z})$.
\item Suppose that $A$ is normed bounded below on $\ell^2({\mathbb N}^-)$. Then $\Phi$ is also normed bounded below.
\end{enumerate}
\end{Lem}

\begin{proof}
\noindent (i). Let $f = (f^-,f^+)\in \ell^2({\mathbb N}^-)\oplus\ell^2({\mathbb N}_0)$. Then $\|f\|^2 = \|f^-\|^2+\|f^+\|^2$ and
$$
\begin{aligned}
\|\Phi f\|^2 = &\|A f^{-}+ Bf^{+}\|^2+ \|f^{+}\|^2\\
 \le& 2(\|A\|^2 \|f^{-}\|^2+\|B\|^2 \|f^{+}\|^2) +\|f^{+}\|^2\\
\le& \max\left( 2\|A\|^2, 2\|B\|^2+1\right) \cdot \left(\|f^{-}\|^2+\|f^+\|^2\right) = C \|f\|^2,
\end{aligned}
$$
where $C =\max\left( 2\|A\|^2, 2\|B\|^2+1\right)$. Thus,  $\Phi$ is bounded. 

\smallskip

\noindent (ii). Since $A$ is norm bounded below. Let $c_1$ be the constant such that
$$
\|Af\|\ge c_1\|f\|,  \ \forall \  f\in \ell^2({\mathbb N}^{-})
$$
Let $c_2 = \|-B\|$.  We now prove that $\Phi$ is norm bounded below.
$$
\begin{aligned}
\|\Phi f\|^2 =& \|A f^-+B f^+\|^2+\|f^{+}\|^2\\
=& \|A f^-+B f^+\|^2+\frac{1}{2} \|f^{+}\|^2+\frac12\|f^{+}\|^2 \\
\ge& \|A f^-+B f^+\|^2+\frac{1}{2c_2} \|-Bf^+\|^2+\frac12\|f^{+}\|^2  \ \\
\ge&  \min\left(1,\frac{1}{2c_2}\right)\cdot\left(\|A f^-+Bf^+\|^2+ \|-Bf^+\|^2\right)+\frac12\|f^{+}\|^2  \\
\ge&  \frac12\min\left(1,\frac{1}{2c_2}\right)\cdot\left(\| Af^-+B f^+\|+ \|-Bf^+\|\right)^2+\frac12\|f^{+}\|^2\\
\ge& \frac12\min\left(1,\frac{1}{2c_2}\right)\|Af^-\|^2+\frac12\|f^{+}\|^2\\
\ge& \min\left(\frac{c_1}{2},\frac{c_1}{4c_2},\frac12\right)\|f\|^2,
\end{aligned}
$$
where we have used the inequality $(a^2+b^2)\ge \frac{1}{2}(a+b)^2$ in the third last line and the triangle inequality in the second last line. This completes the proof.
\end{proof}
Recall that $\Phi_g$ admits a matrix representation of the form
$$\left[
  \begin{array}{ccc}
    T_{\widetilde{g}} & | & H_{\overline{g}} \\
    -- & | & -- \\
    O & |& I \\
  \end{array}
\right]
$$ as in \eqref{rep_Phi}. The above lemma can readily be used once we have the boundedness of the Toeplitz and Hankel operators.  We have thus the following proposition concerning the boundedness of the operators $T_g$.

\begin{Prop}
 ${\mathcal F}(g)$ forms a Bessel sequence for $L^2([-1/2,1/2])$ if and only if $g\in L^{\infty}[-1/2,1/2]$ and $T_g$ is a bounded operator on $\ell^2({\mathbb N}_0)$.
\end{Prop}

\begin{proof}
If $g\in L^{\infty}[-1/2,1/2]$, then both $T_g$ and $H_g$ are bounded operators. Hence, by Lemma \ref{main_lemma}, $\Phi_g$ is a bounded linear operator which is equivalent to the fact that  ${\mathcal F}(g)$ forms a Bessel sequence for $L^2([-1/2,1/2])$  We now give the necessity part with two different proofs from two point of views. One is from the frame theory and the other one is from the Toeplitz operator theory.
\medskip

\noindent{\it (Beurling density proof).} From \cite[Theorem 3.2(i)]{gabardo2014frames}, it was proved that if ${\mathcal F}(g)$ forms a Bessel sequence and $g\not\in L^{\infty}$, then the upper Beurling density of ${\mathbb N}^{-}$ has to be zero. Recall that the upper Beurling density of $\Lambda$ is
$$
D^{+}(\Lambda) = \limsup_{R\rightarrow\infty}\sup_{x\in{\mathbb R}}\frac{\#(\Lambda\cap(x+[-R,R]))}{2R},
$$
From the definition, $D^{+}({\mathbb N}^{-}) = 1$. This is a contradiction. Hence,  $g \in L^{\infty}[-1/2,1/2]$. Thus, $T_g$ is bounded.

\medskip

\noindent{\it (From Toeplitz operator theory)} From the definition of the Bessel sequence, we know that $\|\Phi_gf\|^2\le B\|f\|^2$. Let $f = (f^-,f^+)\in \ell^2({\mathbb N}^-)\oplus\ell^2({\mathbb N}_0)$. Using the representation in (\ref{rep_Phi}),
$$
\|T_{\tilde g}f^-+H_{\overline{g}}f^+\|^2+ \|f^+\|^2\le B(\|f^-\|^2+\|f^+\|^2).
$$
Let $f^+ = 0$, we obtain $\|T_{\tilde g}f^-\|^2\le B\|f^-\|^2$. Hence, $T_{\tilde g}$ is a bounded operator and $g$ is bounded.
\end{proof}

\smallskip

We are now ready to prove Theorem \ref{thmain1}. From the facts about each spectrum, we can formulate Theorem \ref{thmain1} in the following theorem.

\begin{theorem}\label{th2}
\begin{enumerate}
\item  ${\mathcal F}(g)$ is complete if and only if $0\not\in\sigma_{p}(T_{{g}})$.
\item ${\mathcal F}(g)$  forms a frame if and only if $T_g$ is bounded and $0\not\in\sigma_{ap}(T_{{g}})$.
\item ${\mathcal F}(g)$ forms a Riesz basis if and only if $T_g$ is bounded and $0\not\in\sigma(T_{{g}})$.
\end{enumerate}

\end{theorem}

\begin{proof} Because of Lemma \ref{lem1+}, it suffices to prove all the statements with $T_g$ replaced by $T_{\widetilde{g}}$.    Throughout the proof, we write $f = (f^-,f^+)\in \ell^2({\mathbb N}^-)\oplus \ell^2({\mathbb N}_0)$.

\medskip

\noindent(1) Suppose that $0\not\in\sigma_{p}(T_{\widetilde{g}})$. Then $T_{\widetilde{g}}$ is injective. We need to show that $\Phi_g$ is injective. To see this, let $\Phi_g f  =0$ and we have
$$
T_{\widetilde{g}} f^{-} + H_{\overline{g}} f^{+} = 0, \ \mbox{and} \ f^{+}=0.
$$
Hence, $T_{\widetilde{g}} f^{-}=0$. But $T_{\widetilde{g}}$ is injective, $f^{-}=0$. This shows $\Phi_g$ is injective and thus ${\mathcal F}(g)$ is complete. Conversely, suppose that $\Phi_g$ is injective and let $T_{\widetilde{g}}f = 0$. Consider $(f,0)$ and we have $\Phi_g \left(
                  \begin{array}{c}
                    f \\
                    0 \\
                  \end{array}
                \right)= \left(
                  \begin{array}{c}
                    T_{\widetilde{g}}f \\
                    0 \\
                  \end{array}
                \right) =0$. Hence, injectivity of $\Phi_g$ implies that $\left(
                  \begin{array}{c}
                    f \\
                    0 \\
                  \end{array}
                \right) = 0$.   Thus, $ f=0$.
\medskip

\noindent(2). Suppose that $T_g$ is bounded (and thus $g$ is bounded) and  $0\not\in\sigma_{ap}(T_{\widetilde{g}})$. Then the proof of upper bound is trivial by the fact that $g\in L^{\infty}$ since now $T_{\tilde g}$ and $H_{\overline g}$ are bounded. We now show that it satisfies the lower bound. Then we know, $T_{\widetilde{g}}$ is injective and has a closed range. Furthermore, it is norm bounded below. Hence, we can use Lemma \ref{main_lemma} (2) so that $\Phi_g$ is normed bounded below. Thus, Theorem \ref{th1}(2) shows that ${\mathcal F}(g)$ forms a frame. This completes the proof.
%
\medskip

Conversely, suppose that $0\in \sigma_{ap}(T_{\widetilde{g}})$. Then there exists $f_n$ with $\|f_n\|=1$ such that $\|T_{\widetilde{g}} f_n\|\rightarrow 0$ as $n\rightarrow\infty$. Taking the vector $(f_n, 0)$. Then
$$\Phi_g\left(
          \begin{array}{c}
            f_n \\
            0 \\
          \end{array}
        \right) = \left(
                    \begin{array}{c}
                      T_{\overline g} f_n \\
                      0 \\
                    \end{array}
                  \right)\rightarrow 0.
$$
Hence, $\Phi_g$ is not bounded below and hence ${\mathcal F}(g)$ cannot be a frame.

\medskip

\noindent(3).  Suppose that $0\not\in \sigma(T_{\widetilde{g}})$. Then $T_{\widetilde{g}}$ is injective and hence $\Phi_{g}$ is injective by (1). To see that $\Phi_g$ is surjective, we take $\left(
                    \begin{array}{c}
                      y^{-} \\
                      y^{+} \\
                    \end{array}
                  \right)$ and we try to solve
$$
\Phi_g \left(
                    \begin{array}{c}
                      f^{-} \\
                      f^{+} \\
                    \end{array}
                  \right) = \left(
                    \begin{array}{c}
                      y^{-} \\
                      y^{+} \\
                    \end{array}
                  \right).
                  $$
Thus, $f^{+} = y^{+}$ and $T_{\widetilde{g}}f^{-} = y^{-}-H_{\overline{g}}f^{+}$. As $T_{\widetilde{g}}$ is surjective, $T_{\widetilde{g}}^{-1}$ is bounded by inverse mapping theorem of linear operators. Thus, we can find $f^{-} = T_{\widetilde{g}}^{-1}(y^{-}-H_{\overline{g}}f^{+})$ in $\ell^2({\mathbb N})$. Hence, $\Phi_g$ is invertible and possesses a bounded inverse. $0\not\in\sigma(T_{\widetilde{g}})$ follows.

\medskip

Conversely, Suppose that $0\in\sigma(T_{\widetilde{g}})$. Then $0\in \sigma_{ap}(T_{\widetilde{g}})$ or $0\in \sigma_{c}(T_{\widetilde{g}})$. However, if $0\in \sigma_{ap}(T_{\widetilde{g}})$, then (2) shows that ${\mathcal F}(g)$ cannot be a frame and hence cannot be a Riesz bases. On the other hand, if $0\in \sigma_{c}(T_{\widetilde{g}})$, then $T_{\widetilde{g}}$ cannot be surjective. Hence, $\Phi_g$ cannot be surjective either since it was, then $\Phi_g f = (y,0)$ for any $y$ and it makes $T_{\widetilde{g}}$ is surjective.
\end{proof}

\medskip

We will give some geometric characterizations of certain window functions for the frame/basis property of  ${\mathcal F}(g)$. First of all, combining Theorem \ref{th2} and Widom-Devinatz theorem, we have the following theorem.

\begin{theorem}
${\mathcal F}(g)$ forms a Riesz basis for $L^2[-1/2,1/2]$ if and only if $g^{-1}\in L^{\infty}$ and $g$ satisfies the Helson-Szeg\"{o} condition, namely
 \begin{equation}\label{H-S}
 \frac{g}{|g|} = e^{i({u}^h+v+c)} \ \mbox{a.e. on} \ [-1/2,1/2]
 \end{equation}
 where $c\in{\mathbb R}$, $u,v$ are bounded real-valued functions, $\|v\|_{\infty}<\pi/2$ and ${u}^h$ is the harmonic conjugate of $u$, (i.e., $u^h$ satisfies the property that  $u+i{u}^h$ is an analytic function in the unit disk)
\end{theorem}

\medskip

\subsection{Classification of admissible real window functions}
As we have discussed in the introduction, Helson-Szeg\"{o} condition does not give us a directly checkable criterion even for simple classes of functions. Furthermore, there is not even an analogous theorem for $0$ to be outside the point/approximate point spectrum.  We will turn our attention to certain classes of functions.  We first give a complete characterization of the real-valued window functions.

\begin{theorem}\label{threalg}
 Let $g$ be real-valued. Then
 \begin{enumerate}
\item  ${\mathcal F}(g)$ is complete.
\item  The following are equivalent.
\begin{enumerate}
\item  ${\mathcal F}(g)$ forms a Riesz basis.
 \item ${\mathcal F}(g)$ forms a frame.
\item $0\not\in[\mbox{essinf}(g),\mbox{esssup}(g)]$.
\end{enumerate}
\end{enumerate}
\end{theorem}

\begin{proof}
(1). By Theorem \ref{thMR}(2), we know that if $g$ is real-valued, then $T_g$ is injective. Hence, $0\not\in \sigma_p(T_g)$. Hence, by Theorem \ref{th2}, ${\mathcal F}(g)$ is complete.

\medskip

(2) (a) implies (b) is from definition.  (c) implies (a) follows from Theorem \ref{thMR}(iii) which states that  $\sigma(T_{{g}}) = [\mbox{essinf}(g),\mbox{esssup}(g)]$. As $0\not\in\sigma(T_{{g}})$ by (c), Theorem \ref{th2}(3) shows that ${\mathcal F}(g)$ is a Riesz basis.  

\medskip

 We now show (b) implies (c). By Theorem \ref{th2} (2) and Theorem \ref{thMR},  (b) implies that 0 $\not \in \sigma_{ap}(T_{\widetilde{g}})= \sigma_{ap}(T_{g}) \subset \sigma(T_{g})=[\mbox{essinf}(g),\mbox{esssup}(g)]$. This completes the proof.
\end{proof}

\medskip

\begin{Example}\label{example}
{\rm (1). Let $g(x) = x\chi_{[-1/2,1/2]}$. Then ${\mathcal F}(g)$ is complete as $g$ is real-valued. However, it does not form a frame because $0$ is in $[\mbox{essinf}(g),\mbox{esssup}(g)]= [-1,1]$. This example will be used in derivative samplings.}

\medskip

{\rm (2). Let $g(x) = -\chi_{[-1/2,0)}+\chi_{[0,1/2]}$. Then ${\mathcal F}(g)$ is complete as $g$ is real-valued. However, it does not form a frame because $0$ is in the interval $[\mbox{essinf}(g),\mbox{esssup}(g)]= [-1,1]$ (note that essran$(g) = \{-1,1\}$).  On the other hand,   if we consider the windowed exponentials $\{g(x)e^{2\pi i n x}: n\in{\mathbb Z}\}$, then the system forms a frame for $L^2[-1/2,1/2]$  since $|g|=1$.}
\end{Example}

\medskip

\subsection{Classification of admissible complex  exponential functions}
We now turn to study complex windows. As we mentioned before,  it is very difficult to give a  characterization for all complex window functions.   We start with the special case when $g_{\xi}(x) = e^{2\pi i \xi x}$ for some $\xi \in \mathbb{R}$. Then
 $$
 {\mathcal F} (g_{\xi})= \{e^{2\pi i nx}: n=0,1,2,...\}\cup \{e^{2\pi i (\xi+n)x}: n=-1,-2,...\}.
 $$
 
Notice that the set  ${\mathcal F} (g_{\xi})$ consists of complex exponential functions and the problem of determining whether $ {\mathcal F} (g_{\xi})$ is a frame of $L^2[-1/2,1/2]$ is, in fact, a density problem of Fourier frames. Notice that in this case. the lower and upper Beurling density of $\Lambda=\{\mathbb{N}_0, \mathbb{N}^{-}+\xi\}$ are both 1 and hence this problem falls into the interesting gap zone of existing density results (see \cite{casazza2006density}).  We can also see this problem from the point of view of perturbation by noticing that $ {\mathcal F} (g_{\xi})$ is a perturbation version of the standard Fourier basis. To the best of our knowledge, there are no results about the arbitrary shifting of the negative frequency of the standard Fourier basis, and we refer to the reader \cite{acosta2009stability, aldroubi2001nonuniform, balan1997stability} for advances on this direction.

\medskip

Below, we can give a complete characterization for complex exponential functions.

 \begin{theorem}\label{th4}
  \begin{enumerate}
 \item If $\xi<-1/2$, then ${\mathcal F}(g_{\xi})$ is incomplete.
 \item If $-1/2 <\xi <1/2$, then ${\mathcal F}(g_{\xi})$  is a Riesz basis.
 \item If $-1/2+n<\xi<1/2+n, n\in \N^+$,  then ${\mathcal F}(g_{\xi})$ is a frame but not a Riesz basis.
 \item If $\xi=-1/2+n, n\in \N_0$, then ${\mathcal F}(g_{\xi})$ is complete but not a Riesz basis nor a frame.
 \end{enumerate}
\end{theorem}

\begin{proof}
We define a map $T: L^2[-\frac{1}{2}, \frac{1}{2}] \rightarrow L^2 [-\frac{1}{2}, \frac{1}{2}] $ by
 $$
 Tf(x)=e^{-\pi i \xi x/2} f(x).
 $$
  It is easy to see that $T$ is an invertible bounded linear operator. Hence
$\mathcal{F}(g_{\xi})$ is complete/ a Riesz basis/ a frame if and only if $T(\mathcal{F}(g_{\xi}))$ is complete/ a Riesz basis/ a frame. Note that
$$
T(\mathcal{F}(g_{\xi}))=\{ e^{2\pi i (n-\frac{\xi}{2})x}: n\in \N_0\}\cup \{e^{2\pi i (\frac{\xi}{2}+n)x}: n\in \N^-\}.
$$
Now define the real sequence
$$
\Gamma_{\xi}=\left\{ n-\frac{\xi}{2}, n\in \mathbb{N}_0\right\} \cup \left\{ n+\frac{\xi}{2}: n \in \mathbb{N}^-\right\}
$$
 and denote by $E(\Gamma_{\xi})=\{e^{2\pi i \lambda_n \xi}: \lambda_n \in \Gamma_{\xi}\}$ the  complex exponential system associated with $\Gamma_{\xi}$. Then we have $T(\mathcal{F}(g_{\xi})) = E(\Gamma_{\xi})$ and it is sufficient to determine whether $E(\Gamma_{\xi})$ is complete/ a Riesz basis/ a frame. 
 
 \begin{enumerate}
 \item Following the idea of proof of Theorem V in \cite{levinson1940gap}, we let $t = -\xi/2$ and let $F_t(x) = \sin (\pi x)\cos^{2t-1}(\pi x)$. Then by calculation, $F_t\in L^2[-1/2,1/2]$ if and only if $t>1/4$.  Hence if $\xi<-1/2$, then $F_t\in L^2[-1/2,1/2]$.   For $n\ge 0$,  
$$\begin{aligned}
& \int_{-1/2}^{1/2} F_t(x) e^{-2\pi i (n+t)x}dx \\=& \frac{1}{2^{2t}}\int_{-1/2}^{1/2} (e^{i\pi x}-e^{- i\pi x}) (e^{\pi x}+e^{- i\pi x})^{2t-1}  e^{-2\pi i (n+t)x}\\
=&\int_{-1/2}^{1/2} (e^{2\pi i x}-1)  (1+e^{-2\pi i x})^{2t-1}  e^{-2\pi i nx}dx\\
=&\lim_{r\rightarrow 1^{+}} \int_{-1/2}^{1/2} (e^{2\pi i x}-1)  (1+ re^{-2\pi i x})^{2t-1}  e^{-2\pi i nx}dx\\
=&\lim_{r\rightarrow 1^{+}} \int_{-1/2}^{1/2} (e^{2\pi i x}-1) \sum_{k=0}^{\infty} \binom{2t-1}{k}r^k e^{-2\pi ikx}  e^{-2\pi i nx}dx\\
=&\lim_{r\rightarrow 1^{+}} \sum_{k=0}^{\infty} \binom{2t-1}{k}r^k\int_{-1/2}^{1/2} (e^{2\pi i x}-1)  e^{-2\pi i(k+n)x}  dx =0.\\
\end{aligned}
$$
Similarly,  we can also show that $\int_{-1/2}^{1/2} F_t(x) e^{-2\pi i (n-t)x}dx =0$.  It follows that  $E(\Gamma_{\xi})$ is not complete for $\xi<-1/2$.
 \item If $-1/2<\xi<1/2$, then $ E(\Gamma_{\xi})$ is a Riesz basis followed by  the classicial Kadec-1/4 theorem.   
 
 \item Suppose that  $\xi =\xi_0+n$ for $\xi_0 \in (-1/2,1/2)$ and $n$ is a positive integer. Then the original system $\mathcal{F}(g_{\xi_0}) \subsetneq \mathcal{F}(g_{\xi}) $. Note that we have proved that $\mathcal{F}(g_{\xi_0})$ is a Riesz basis by (2), hence $\mathcal{F}(g_{\xi})$ is a frame, but not a Riesz basis since there are excessive elements.
 
 \item First, we can use \cite[Theorem 4 and Section 8]{redheffer1983completeness}) to prove that  $E(\Gamma_{-1/2})$ is exact (i.e., minimal and complete) but not a Riesz basis.  Then we know $E(\Gamma_{-1/2})$ is not a frame either. Otherwise, $E(\Gamma_{-1/2})$ will be an exact frame, which means it is a Riesz basis and results in a contradiction. Hence, we conclude that $\mathcal{F}(g_{-1/2})$ is complete but not a Riesz basis nor a frame. In general, note that we have
 $$
 \mathcal{F}(g_{-1/2}) \subset \mathcal{F}(g_{1/2}) \subset  \mathcal{F}(g_{3/2}) \subset \cdots.
 $$
 Suppose that there exists a set in this chain forms a frame.  Then we can remove finitely many elements from this set to obtain $\mathcal{F}(g_{-1/2})$. Note that the removal of finitely many elements from a frame leaves either a frame or an incomplete set. Now, $ \mathcal{F}(g_{-1/2})$ is complete, so if one of the sets forms a frame, then $ \mathcal{F}(g_{-1/2})$ is a frame also. This is a contradiction to what we just proved. Hence,
 every set in the above chain is complete and is not a Riesz basis nor frame.
 \end{enumerate}

\end{proof}

\section{Applications}

\noindent{\bf Dynamical Sampling. } Problems in determining the frame property of ${\mathcal F}(g)$  can be applied to dynamical sampling problems. Formulated mathematically,  Aldroubi and his collaborators \cite{aceska2016scalability, aceska2014dynamical, aldroubi2017iterative, aldroubi2015dynamical, aldroubi2013dynamical,  aldroubi2015exact, philipp2017bessel} considered $F\in \ell^2({\mathbb Z})$ and the spatial sampling sets  $\{\Omega_i\}$, $i=1,...,L$, are collections of proper subsets in ${\mathbb Z}$. They are usually assumed to be nested, i.e., $ \Omega_1 \subset \cdots \subset \Omega_L$, and are well-spread. For example, $\Omega_1=\cdots \Omega_L=m\mathbb{Z}  $ $(m>1)$ in \cite{ aldroubi2013dynamical}. Let also $G_1,...,G_L\in \ell^2({\mathbb Z})$\footnote{To avoid confusion, we will use upper case letter to denote functions in $\ell^2({\mathbb Z})$ and Paley-Wiener space $PW=\{f\in L^2(\mathbb{R}): supp(\hat f) \subset [-1/2, -1/2] \}$ and lower case letter to denote functions in $L^2[-1/2,1/2]$ }. We want to reconstruct $F$ from the samples
$$
(G_1\ast F(\Omega_1),....,G_L\ast F (\Omega_L))
$$
(Here $G(\Omega_i)=\{G(n): n\in \Omega_i\}$). Notice that this model is equivalent to sampling through the Paley-Wiener space of bandlimited functions supported on a bandwidth of length one as well as certain shift-invariant spaces \cite{aldroubi2013dynamical}. As $\Omega_i$ are proper subsets of ${\mathbb Z}$, we may think they are below the Nyquist rate and hence samples on an individual $\Omega_i$ are not able to recover $F$.

\medskip

We say that $(G_i,\Omega_i)_{i=1,...,L}$ allows a {\it stable sampling} if the linear map
$$
{\mathbf A}: \ell^2({\mathbb Z})\rightarrow \bigoplus_{i=1}^{L}\ell^2(\Omega_i)
$$ defined by
\begin{equation}\label{A_map}
{\mathbf A}F = (G_1\ast F(\Omega_1),....,G_L\ast F (\Omega_L))
\end{equation}
admits a bounded left-inverse ${\mathbf B}$ such that ${\mathbf B}{\mathbf A}F = F$ for all $F\in \ell^2({\mathbb Z})$. We define the Fourier transform of $G\in \ell^2({\mathbb Z})$ by
$$
g(x): = \widehat{G}(x) = \sum_{n=-\infty}^{\infty} G(n)e^{-2\pi i nx}, x\in{\mathbb T},
$$
where ${\mathbb T} $ is the circle group identified as $ [-1/2,1/2]$. We first notice that for  ${\mathbf A}$ to be well-defined, one  should naturally require $G_i\ast F\in \ell^2({\mathbb Z})$. In fact, this is equivalent to $\widehat{G_i}\in L^{\infty}[-1/2,1/2]$. The following standard theorem connects dynamical sampling problems with windowed exponentials we have been considering. This can be seen directly by taking Fourier transform, so we will omit its proof.

\begin{theorem}\label{th0.3}
Let $G_1,...,G_L\in\ell^2({\mathbb Z})$ be sequences such that $\widehat{G_i}\in L^{\infty}$ and let $g_i(\xi) = \widehat{G_i}(\xi)$. Then $(G_i,\Omega_i)_{i=1,...,L}$ allows a stable sampling if and only if $\bigcup_{i=1}^{L}{\mathcal E}(\overline{g_i},\Omega_i)$ forms a frame of windowed exponentials on $L^2[-1/2,1/2]$.
\end{theorem}

Because of this theorem, the frame property of ${\mathcal F}(g)$  allows us to determine the set of all $G$ such that
$$
\{F(\N^{-}), (G\ast F)(\N_{0})\} 
$$
is stably recoverable (note that $F(\N_{-}) = (\delta_0\ast F)(\N^{-})$ with $\delta_0$ is the Dirac function on $\ell^2({\mathbb Z})$). In contrast with $\Omega_i = m{\mathbb Z}$ in \cite{aldroubi2013dynamical}. Notice that  ${\mathbb N}^{-}$ and ${\mathbb N}_0$ has zero lower Beurling densities. This means that each individual samples $F(\N^{-})$, $ (G\ast F)(\N_{0})$ could not admit a stable recovery.
\medskip

\noindent{\bf Derivative Sampling.}  Uniform sampling of derivatives is a classical topic in sampling theory, see \cite{fogel1955note, jagerman1956some, linden1960generalization, papoulis1977generalized, rawn1989stable, zibulski1995frame} and references therein. However, relatively few papers have considered nonuniform sampling with derivatives. We refer to \cite{adcock2017density, grochenig1992reconstruction, razafinjatovo1994iterative} for the study of density of nonuniform derivative samples that give stable reconstructions.  In this paper, we provide new examples of nonuniform derivative sampling.  Consider the sampling on Paley-Wiener space $PW$ by the samples $\{F(n)\}_{n\ge 0}\cup\left\{\frac{-1}{2\pi i}F'(n)\right\}_{n<0}$. We note that if $f = \widehat{F}$, we have
$$
F'(\xi) =  \frac{d}{dx}\int_{-1/2}^{1/2} f(x)e^{-2\pi i \xi x }dx = \int_{-1/2}^{1/2} f(x) (-2\pi i x)e^{-2\pi i \xi x }dx
$$
Hence, the frame/basis and completeness properties are equivalent to that of the windowed exponentials ${\mathcal E}(\chi_{[-1/2,1/2]},{\mathbb N}_0)\cup {\mathcal E}( x\chi_{[-1/2,1/2]},{\mathbb N}^-)$. By Example \ref{example}(1), this collection is complete, but not forming a frame. Hence,
 
 \begin{theorem}
 The one-sided derivative sampling by  $\{F(n)\}_{n\ge 0}\cup\left\{\frac{-1}{2\pi i}F'(n)\right\}_{n<0}$ is not stable but it is injective.
 \end{theorem}

However, if we restrict to a subclass of bandlimited function, we are still able to establish the stable sampling result.

\medskip

\begin{Prop}
Let $E = \{F\in PW: F(x)= F(-x)\}$ be the subspace of even functions on $PW$. Then there exists $c>0$ such that
$$
\sum_{n\ge0}|F(n)|^2+\sum_{n<0}\left|\frac{-1}{2\pi i}F'(n)\right|^2\ge c \|F\|^2, \forall F\in E.
$$
\end{Prop}

\begin{proof}
Note that $F$ is an even function if and only if $f = \widehat{F}$ is an even function. Therefore, it suffices to show that ${\mathcal E}(\chi_{[-1/2,1/2]},{\mathbb N}_0)\cup {\mathcal E}( x,{\mathbb N}^-)$ forms a windowed exponential on the subspace $\widehat{E} = \{f\in L^2[-1/2,1/2]: f(x) = f(-x)\}$. i.e. There exists $c>0$ such that
$$
\sum_{n\ge 0}|\widehat{f}(n)|^2 +\sum_{n\le -1}|\widehat{ xf}(n)|^2 \ge c\|f\|^2, \forall f\in \widehat{E}.
$$
(recall that $e_n(x) = e^{2\pi i nx}$). Throughout the proof, we let $I = [-1/2,1/2]$. By the Parseval's identity,
$$
\sum_{n\ge 1}|\widehat{f}(n)|^2+\sum_{n\le -1}|\widehat{f}(n)|^2 = \|f\|^2-(\widehat{f}(0))^2  =\int_I|f|^2 - \left|\int_I f\right|^2.
$$
As $\widehat{f}(n) = \widehat{f}(-n)$, the sum on the left hand side are the same. Thus,
\begin{equation}\label{eq6.3}
\sum_{n\le -1}|\widehat{f}(n)|^2 = \frac{\int_I|f|^2 - |\int_I f|^2}{2}.
\end{equation}
Similarly, we notice that $ xf$ is an odd function, so we have
\begin{equation}\label{eq6.4}
\sum_{n\le -1}|\widehat{ xf}(n)|^2 = \frac{\int_I\left|xf \right|^2 - \left|\int_I xf\right|^2}{2}.
\end{equation}
Now, we write
\begin{equation}\label{eq.65}
\sum_{n\ge 0}|\widehat{f}(n)|^2 +\sum_{n\le -1}|\widehat{ xf} (n)|^2  = \|f\|^2-\left(\sum_{n\le -1}|\widehat{f}(n)|^2-\sum_{n\le -1}| \widehat{ xf}(n)|^2\right)
\end{equation}
By (\ref{eq6.3}) and (\ref{eq6.4}),
$$
\begin{aligned}
\sum_{n\le -1}|\widehat{f}(n)|^2-\sum_{n\le -1}|\widehat{ xf}(n)|^2=&\frac{\int_I|f|^2 - \int_I | xf|^2-\left(\left|\int_I f\right|^2- \left|\int_I  x f\right|^2\right)}{2}\\
\le& \frac{\int_I|f|^2+\left|\int_I x f\right|^2}{2}\\
\le &\frac{\int_I|f|^2+\left(\int_I x^2 \right)\cdot\left(\int_I |f|^2\right)}{2} = \frac{13}{24}\|f\|^2.
\end{aligned}
$$
Putting back to (\ref{eq.65}), we obtain
$$
\sum_{n\ge 0}|\widehat{f}(n)|^2 +\sum_{n\le -1}|\widehat{ xf}(n)|^2\ge \|f\|^2-\frac{13}{24}\|f\|^2 = \frac{11}{24}\|f\|^2.
$$
\end{proof}

\section{Conclusions}
In this paper, we provide a  study on a specific class of under-sampled windowed exponentials. By making connections to Toeplitz operator, we show for real-valued $g$, the completeness/frame/basis property of $$ {\mathcal F}(g): =\{e^{2\pi i n x}: n\ge 0\}\cup \{g(x)e^{2\pi i nx}: n<0\}$$
 only depends on the essential range of $g$. When $g=e^{2\pi i \xi x}$, we  completely characterized  $\xi$ such that the set
 $\mathcal{ F}(g)$ is complete/frame/basis, which complements those existing results in the literature (e.g. \cite{casazza2006density}). The theory we developed also provides nontrivial new examples of dynamical sampling and nonuniform derivative sampling.  These results shed some light on the interesting connection between the frame theory of windowed exponentials and the Toeplitz operators.

 Our interest, in general, will be for a given countable set of $\Lambda_1,...,\Lambda_N$, can we develop some necessary and sufficient condition to determine $g_1,...,g_N$  so that the windowed exponential
$$
\bigcup_{j=1}^N{\mathcal E}(g_j,\Lambda_j)
$$
is complete or forms a frame/basis on $L^2[-1/2,1/2]$ (or any $L^2(\Omega)$)? Particularly when $\Lambda_j$ are under the critical sampling rate, the classification of $g_j$ is an interesting problem.  Furthermore,  it is also immediate to see any such classification will lead to a new sampling scheme in dynamical samplings and derivative samplings. So far, the classification is only available when $\Lambda = m{\mathbb Z}$ (\cite{aldroubi2015exact}) and the result we presented in this paper. We will leave more classes of undersampled windowed exponentials with samples taken from subgroups or semi-groups for future study.

\section*{Acknowledgement}
We would like to thank anonymous reviewers for their very helpful comments. Sui Tang is supported by the AMS Simons travel grant.

\medskip

\bibliography{referencesLT}{}
\bibliographystyle{siam}

\end{document}